\documentclass[a4paper, 12pt]{amsart}
\usepackage[latin1]{inputenc}
\usepackage[ T1]{fontenc}
\usepackage[english]{babel}
\usepackage{amssymb}
\usepackage{amsmath}
\usepackage{amsthm}
\usepackage{amscd}
\usepackage{amsfonts}
\usepackage{stmaryrd}
\usepackage{pb-diagram}
\usepackage{epic,eepic,epsfig}
\usepackage{a4wide}
\usepackage{nextpage}
\usepackage{fancyhdr}
\usepackage{enumerate}

\newtheorem{prop}{Proposition}[section]
\newtheorem{corollary}[prop]{Corollary}
\newtheorem{lemma}[prop]{Lemma}
\newtheorem{thm}[prop]{Theorem}

\newtheorem{theorem}[prop]{Theorem}
\newtheorem{corol}[prop]{Corollary}
\theoremstyle{definition}
\newtheorem{rem}[prop]{Remark}

\newtheorem{definition}[prop]{Definition}

\begin{document}

\title{The regulator dominates the rank}
\author{{F}abien {P}azuki}
\address{Fabien Pazuki. University of Copenhagen, Institute of Mathematics, Universitetsparken 5, 2100 Copenhagen \O, Denmark, and Universit\'e de Bordeaux, IMB, 351, cours de la Lib\'eration, 33400 Talence, France.}
\email{fpazuki@math.ku.dk}

\thanks{We believe that the results presented here would have been to the liking of Alexey Zykin, who is deeply missed. In loving memory of Alexey and Tanya. We thank Pascal Autissier and Marc Hindry for interesting conversations. We thank the referee for useful feedback. We thank the Swedish Research Council under grant no. 2016-06596, as this work was finalized while the author was in residence at Institut Mittag-Leffler in Djursholm, Sweden during the fall of 2021. The author is supported by ANR-17-CE40-0012 Flair and ANR-20-CE40-0003 Jinvariant.}
\maketitle

\noindent \textbf{Abstract.}
After noticing that the regulator of a number field dominates the rank of its group of units, we bound from below the regulator of the Mordell-Weil group of elliptic curves over global function fields of characteristic $p\geq5$. The lower bound is an increasing function of the rank and of the height. This partially answers Question 7.1 and Question 7.2 of \cite{AHP}. 

{\flushleft
 }

\begin{center}
---------
\end{center}

\thispagestyle{empty}

\maketitle

\section{Introduction}\label{section def}

Regulators of number fields and regulators of Mordell-Weil groups of abelian varieties have attracted a lot of attention, both for their own sake, and for the role they play in the Class Number Formula and in the strong form of the Birch and Swinnerton-Dyer conjecture, respectively. When studying families of number fields or families of abelian varieties, it is sometimes necessary to estimate the size of the regulator in terms of easier invariants, like the discriminant and degree of the number fields, or like the height of the abelian varieties and the rank of their Mordell-Weil group, respectively. In this note, we propose a new lower bound on the regulator of elliptic curves defined over global function fields. This lower bound is an increasing function of the rank of the elliptic curve (when the height is big enough), which is a new phenomenon, and which mirrors a similar situation taking place between the regulator of a number field and its rank of units. We describe both results in the rest of this introduction.

\subsection{Regulators and ranks of units of number fields}

Let us start with the following theorem, which has been an important motivation for this work. In the sequel, if $F$ is a number field, we denote by $d$ its degree over $\mathbb{Q}$. Let $r_1$ be the number of real embeddings of $F$, and $r_2$ be the number of pairs of complex conjugate embeddings of $F$. The group of units of $F$ is a $\mathbb{Z}$-module of finite rank, we denote this rank by $r_F$. Let $R_F$ be the regulator of $F$, and let $w_F$ be the number of roots of unity in $F$.

\begin{thm}(Friedman, \cite{Fri} page 620, Corollary)
Let $F$ be a number field. Then
\begin{equation}\label{ineq1}
\frac{R_F}{w_F}\geq 0.0031\exp(0.241d+0.497r_1).
\end{equation}
\end{thm}

By Dirichlet's unit theorem, we know that $r_F=r_1+r_2-1$. We also know that $d=r_1+2r_2$. This has the following easy consequence when used in inequality (\ref{ineq1}).

\begin{corol}
Let $F$ be a number field. Then
\begin{equation}\label{ineq2}
R_F\geq 0.0062\exp(0.241 r_F)
\end{equation}
\end{corol}

So the story begins with the following fact given by inequality (\ref{ineq2}): the regulator $R_F$ of a number field $F$ dominates the rank $r_F$ of its group of units. This triggers questions about other contexts, for instance: to what extent would the regulator of the Mordell-Weil group of an abelian variety over a global field dominate the rank of this Mordell-Weil group? 

\subsection{Elliptic curves and ranks of Mordell-Weil groups}

Our goal is to prove that the regulator $\mathrm{Reg}(E/K)$ of an elliptic curve $E$ defined over a function field $K$ of characteristic $p\geq 5$ dominates the rank of its Mordell-Weil group. In doing so we partially answer Question 7.1 and Question 7.2 of \cite{AHP} in the case where $K=\mathbb{F}_q(\mathcal{C})$ is a function field of characteristic $p\geq 5$, where $\mathcal{C}$ is a smooth projective and geometrically connected curve defined over its constant field $\mathbb{F}_q$ and of genus $g\geq0$. Note that the rank of elliptic curves over function fields of positive characteristic is not bounded \cite{Ulm, Griff}, hence this improvement is non-trivial. Let us state the result.


\begin{theorem}\label{reg>rank}
Let $K=\mathbb{F}_q(\mathcal{C})$ be a function field of characteristic $p\geq 5$ and genus $g$. Let $E$ be an elliptic curve over $K$ of discriminant $\Delta(E/K)$, of trace zero, and let $p^s$ denote the inseparability degree of the $j$-map of $E$. Let $r$ denote the rank of $E(K)$. There exists a positive real number $c_0=c_0(q,g,p^s)$ such that 
\begin{equation}\label{minoration}
\mathrm{Reg}(E/K)\geq \Big(c_0 \log 12h(E)\Big)^r,
\end{equation}
where $h(E)=\frac{1}{12}\deg \Delta(E/K)$, and the inequality holds with the explicit value $$c_0=\Big(p^{2s} 12 \sqrt{q} (\log q)^2 (5g+9)10^{15.5+23g}\Big)^{-1}.$$
\end{theorem}

We can now deduce the following corollary, which can be seen as a refined Northcott property for the regulators of elliptic curves over function fields in characteristic $p\geq5$.

\begin{corollary}\label{Northcott}
Let $K=\mathbb{F}_q(\mathcal{C})$ be a function field of characteristic $p\geq5$ and genus $g$. The set of elliptic curves of trace zero over $K$, with positive rank, bounded inseparability degree and bounded regulator is finite.
\end{corollary}

\begin{rem}
Under the ABC conjecture, the BSD conjecture, and the GRH, one obtains an inequality for elliptic curves over number fields similar to the inequality (\ref{minoration}) using \cite{Mes82}. This would lead to an improvement of Theorem 4 page 1124 of \cite{Paz16}, as the regulator would bound the rank from above.
\end{rem}

The rest of the text presents a proof of Theorem \ref{reg>rank} and of Corollary \ref{Northcott}. After giving the prerequisites in the next section, we prove inequality (\ref{minoration}). The proof relies on the Minkowski successive minima inequality, combined with a lower bound on the canonical height of non-torsion points on elliptic curves. This is not enough, though, we need extra input to obtain the correct dependance in the rank. We are then able to give an explicit estimate on the analytic rank in Lemma \ref{BruPaz}, following Brumer's work, and transfer this estimate on the algebraic rank via Tate's work. The estimate is of sufficient quality to yield the result.

\section{Definitions and prerequisites}

Here we gather the basic definitions --function fields, heights, regulators of elliptic curves-- and the key results used later in the proof of Theorem \ref{reg>rank}.

\subsection{Function fields} Let $K=k(\mathcal{C})$ be the function field of a smooth projective and geometrically connected curve $\mathcal{C}$ defined over its constant field $k$ and of genus $g\geq0$. Let $M_K$ stand for a complete set of inequivalent valuations $v(.)$. The set $M_K$ is in bijection with the set of closed point in $\mathcal{C}$. Given a place $v\in{M_K}$, the residue field $k_v$ of $K$ at $v$ is a finite extension of $k$: the degree $n_v:= [k_v:k]$ of this extension will be called the degree of $v$.

This gives a normalization such that for any element $x\in{K}$, $x\neq0$, the following product formula holds  $$ \sum_{v\in{M_K}}n_v v(x)=0.$$ 

A divisor $I$ on the field $K$ is a formal sum $\displaystyle{\sum_{v\in{M_K}}a_v \cdot v}$ where $a_v\in{\mathbb{Z}}$ is zero for all but finitely many places $v$. We pose $$\deg(I)=\sum_{v\in{M_K}}n_v a_v.$$

We define the height on $K$ by $h(0)=0$ and for any non-zero $x\in{K}$, by $$h(x)=\sum_{v\in{M_K}}n_v \max\{0, -v(x)\}.$$

If we now consider $E$ to be an elliptic curve defined over the function field $K$, we define the N\'eron-Tate height on the group of rational points $E(K)$ with respect to the divisor $(O)$ on $E$ by $$\widehat{h}_E(P)=\frac{1}{2}\lim_{n\to\infty}\frac{1}{n^2}h(x([n]P)).$$

\subsection{Regulators of elliptic curves}

Let $K$ be a function field of transcendence degree one over its field of constants $k$.
Let $E/K$ be an elliptic curve over the field $K$. We assume that $E$ has trace zero. Let $m$ be the Mordell-Weil rank of $E(K)$, which is finite by the Lang-N\'eron theorem, see \cite{Co06} for instance. Let $\widehat{h}_{E}$ be the N\'eron-Tate height on $E$. Let $<.,.>$ be the associated bilinear form, given by $$<P,Q>=\frac{1}{2}\Big(\widehat{h}_{E}(P+Q)-\widehat{h}_{E}(P)-\widehat{h}_{E}(Q)\Big)$$ for any $P,Q\in{E(K)}$.

\begin{definition}\label{reg abvar}
Let $P_1, ..., P_{r}$ be a basis of the lattice $E(K)/E(K)_{\mathrm{tors}}$, where $E(K)$ is the Mordell-Weil group. The regulator of $E/K$ is defined by $$\mathrm{Reg}(E/K)= \det(<P_i,P_j>_{1\leq i,j\leq r}).$$ In the case $r=0$, the regulator is equal to $1$.
\end{definition}

We gather here three results needed for the sequel.

\begin{lemma}\label{success} (Lemma 3.1 of \cite{AHP})
Let $K$ be a function field of transcendence degree one over its field of constants $k$. Let $E$ be an elliptic curve over the field $K$. We assume that $E$ has trace zero. Let $r$ be the Mordell-Weil rank of $E(K)$. Assume $r\geq 1$. Let $\Lambda=E(K)/E(K)_{\mathrm{tors}}$ and for any $i\in\{1, ..., r\}$, let us denote the Minkowski $i$th-minimum of $(\Lambda, \sqrt{\widehat{h}_{E}})$ by $\lambda_i$. Then we have
\begin{equation}\label{Minkowski2}
 \lambda_1 \cdots \lambda_{r}\leq r^{{r/2}}  (\mathrm{Reg}(E/K))^{1/2}.
\end{equation}
\end{lemma}

\begin{theorem}(Theorem 6.1 of \cite{AHP})\label{Langp}
Let $K=k(\mathcal{C})$ be a function field of characteristic $p>0$ and genus $g$. Let $E/K$ be an elliptic curve of discriminant $\Delta(E/K)$ and assume that the $j$-map of $E$ has inseparable degree $p^s$. Let $P\in{E(K)}$ be a non-torsion point. Then one has $$\widehat{h}_E(P)\geq p^{-2s} 10^{-15.5-23g} h(E),$$ where $h(E)=\frac{1}{12}\deg \Delta(E/K)$.
\end{theorem}

\begin{lemma}\label{ranks}
Let $K=k(\mathcal{C})$ be a function field of characteristic $p>0$ and genus $g$. Let $E/K$ be an elliptic curve over $K$. Let $r_{\mathrm{an}}$ denote the analytic rank of $E/K$ and let $r$ denote its algebraic rank over $K$. Then $r\leq r_{\mathrm{an}}$.
\end{lemma}

\begin{proof}
This is a direct consequence of Theorem 5.2 page 436 of \cite{Tate}.
\end{proof}

\section{Regulators of elliptic curves over function fields of positive characteristic}\label{ellFFp}

Let us start with a useful lemma, which is an explicit version of Proposition 6.9 page 463 in \cite{brumer}. Inequality (\ref{weakeasy}) is weaker than inequality (\ref{long}), but easier to manipulate. Brumer's work \cite{brumer} provides a bound on the analytic rank. To deduce the control on the algebraic rank we use Lemma \ref{ranks}. 

\begin{lemma} \label{BruPaz}
Let $K=\mathbb{F}_q(\mathcal{C})$ be a function field of characteristic $p\geq 5$ and genus $g$. Let $E$ be an elliptic curve over $K$. Let $n_E$ be the degree of the conductor of $E$ and let $r$ denote the rank of $E(K)$. Assume $n_E>1$. The following inequality holds:
\begin{equation}\label{long}
r\leq \frac{n_E}{2\log n_E}\log q + \frac{n_E}{(\log n_E)^2}4\sqrt{q}(\log q)^2+\frac{7}{2}+\frac{\log q}{\sqrt{q}\log n_E}\Big[ (2g-2)\sqrt{q}+ 20g+17 \Big],
\end{equation}
and leads to, as $q\geq 5$, 
\begin{equation}\label{weakeasy}
r\leq \frac{n_E}{\log n_E} \sqrt{q} (\log q)^2 (5g+9).
\end{equation}
\end{lemma}

Note that we do not assume that $n_E$ is large when compared to $q$, in contrast with Proposition 6.9 page 463 of \cite{brumer}.

\begin{proof}
We follow closely the proof of Proposition 6.9 page 463 in \cite{brumer}. Let us denote by $\mathcal{Z}$ the set of $\theta$ such that $1+i\theta/\log q$ is a zero of the $L$-function of the elliptic curve $E$, and such that $0\leq \theta<2\pi$. For any trigonometric polynomial $f$ with Fourier coefficients denoted $c(n)$, we state the explicit formula (6.7) page 462 in \cite{brumer}, for a positive integer parameter $Y$ to be fixed later:
\begin{equation}\label{explicit}
r_{\mathrm{an}} f(0)+\sum_{\theta\in\mathcal{Z}}f(\theta)=c(0)(n_E+4g-4)+2\sum_{m=1}^{Y}U_m(E,f),
\end{equation}
where the $U_m(E,f)$ are defined in (6.6) page 462 in \cite{brumer} and satisfy the following inequality, uniformly in $f$ (there is a term $\beta_K$ in the original formula, note that we used $\beta_K=(2g+1)(1-q^{-1})^{-1}$, as given in Proposition 6.3 page 461 of \cite{brumer}):
\begin{equation}\label{m>2}
\sum_{m=3}^{+\infty}\vert U_m(E,f)\vert\leq \frac{2}{\sqrt{q}(1-q^{-1/2})^2}+\frac{(4g+2)(1-q^{-1})^{-1}}{(q-1)(1-q^{-1/2})},
\end{equation}

\noindent and if we consider (as in (6.5) page 465) the F\'ejer kernel given by $$F_Y(\theta)=\frac{(\sin \frac{1}{2}Y\theta )^2}{Y (\sin\frac{1}{2}\theta)^2},$$
for the specific choice $f=F_Y$, we have the inequality
\begin{equation}\label{m=2}
\vert U_2(E,F_Y)\vert \leq \frac{Y}{2}+\frac{4g+2}{(\sqrt{q}-1)(1-q^{-1})},
\end{equation}

\noindent and the inequality

\begin{equation}\label{m=1}
\vert U_1(E,F_Y)\vert\leq \frac{2q^{Y/2}}{\sqrt{q} Y (1-q^{-1/2})^2} + \frac{(2g+1)}{(1-q^{-1})} Y.
\end{equation}
Following Brumer we fix $f=F_Y$ in equation (\ref{explicit}), we get $f(0)=Y$ and $c(0)=1$, and because the F\'ejer kernel is \textbf{non-negative}\footnote{The author remembers attending a course in functional analysis of Jean-Michel Morel at ENS Cachan in 2002, where one needed to compare different kernels in Fourier theory. The F\'ejer kernel will always be remembered as one of the most important, because it is non-negative, this is useful again in this situation!}, the combination of (\ref{explicit}) with (\ref{m>2}), (\ref{m=2}), (\ref{m=1}) leads to 

\begin{equation}\label{Y}
\begin{tabular}{ll}
$r \leq r_{\mathrm{an}}\leq$ & $\displaystyle{\frac{n_E+4g-4}{Y}+\frac{4q^{Y/2}}{Y^2\sqrt{q}(1-q^{-1/2})^2}+\frac{4g+2}{1-q^{-1}}+1+\frac{(8g+4)}{Y(\sqrt{q}-1)(1-q^{-1})}}$\\
\\
$$ & $\displaystyle{+\frac{4}{Y\sqrt{q}(1-q^{-1/2})^2}+\frac{(8g+4)(1-q^{-1})^{-1}}{Y(q-1)(1-q^{-1/2})}}.$\\
\\
\end{tabular}
\end{equation}

For any $n_E>1$, we may now fix $Y=\lceil \frac{2\log n_E}{\log q} \rceil>0$ and use $\frac{1}{Y}\leq \frac{\log q}{2\log n_E}$ and $q^{Y/2}\leq n_E\, q$ to obtain\footnote{Taking $\lfloor\cdot\rfloor$ instead of $\lceil\cdot\rceil$ when choosing $Y$ is a valid option if $n_E$ is assumed big when compared to $q$.} in (\ref{Y})
\\

\begin{tabular}{ll}
$r \leq$ & $\displaystyle{\frac{n_E}{2\log n_E}\log q+\frac{n_E \sqrt{q}(\log q)^2}{(\log n_E)^2(1-q^{-1/2})^2}+\frac{4g+2}{1-q^{-1}}+1}$\\
\\
$$ & $\displaystyle{+\frac{\log q}{\log n_E}\Big[ \frac{4g+2}{(\sqrt{q}-1)(1-q^{-1})} + \frac{2}{\sqrt{q}(1-q^{-1/2})^2} + \frac{(4g+2)(1-q^{-1})^{-1}}{(q-1)(1-q^{-1/2})} +2g-2\Big]},$\\
\\
\end{tabular}

\noindent and with $q\geq5$ we obtain
\begin{equation}
r\leq \frac{n_E}{2\log n_E}\log q + \frac{n_E}{(\log n_E)^2}4\sqrt{q}(\log q)^2+\frac{7}{2}+\frac{\log q}{\sqrt{q}\log n_E}\Big[ 20g+17 +(2g-2)\sqrt{q} \Big].
\end{equation}

%
%
%
%
%

\end{proof}

We can now give the proof of Theorem \ref{reg>rank}.

\begin{proof} 
If $r=0$ the result is obvious, we may thus assume that $r\geq1$. We start by combining Lemma \ref{success} and Theorem \ref{Langp} to obtain
\begin{equation}\label{success+Langp}
\mathrm{Reg}(E/K)\geq \frac{1}{r^r} \Big( p^{-2s} 10^{-15.5-23g} h(E)\Big)^r.
\end{equation}

We now want to estimate the denominator by bounding the algebraic rank $r$ from above: one uses Lemma \ref{BruPaz} (valid when $p\geq 5$ and $n_E>1$):
\begin{equation}\label{Brumer}
r\leq \frac{n_E}{\log n_E} \sqrt{q} (\log q)^2 (5g+9).
\end{equation}

\noindent Now, as $n_E\leq 12 h(E)$ and as $x\mapsto x(\log x)^{-1}$, for $x>e$ is a well defined increasing function, one deduces from (\ref{Brumer}) that for $n_E>e$ 

\begin{equation}
r\leq \frac{12h(E)}{\log 12h(E)} \sqrt{q} (\log q)^2 (5g+9),
\end{equation}
which leads to 
\begin{equation}
\mathrm{Reg}(E/K)\geq \Big(\frac{\log 12h(E)}{12h(E) \sqrt{q} (\log q)^2 (5g+9)}\Big)^{r} \Big( p^{-2s} 10^{-15.5-23g} h(E)\Big)^r 
\end{equation}
and finally
\begin{equation}
\mathrm{Reg}(E/K) \geq  \Big( c_0 \log 12h(E)\Big)^r,
\end{equation}
where $c_0=\Big(p^{2s} 12 \sqrt{q} (\log q)^2 (5g+9)10^{15.5+23g}\Big)^{-1}$. We also need to treat the case $n_E\leq e$: we may use the easy bound $r\leq n_E+4g-4$, which gives in particular $r\leq 4g-1$. Inject this in (\ref{success+Langp}) to obtain

\begin{equation}
\mathrm{Reg}(E/K)\geq  \Big( \frac{1}{4g-1} p^{-2s} 10^{-15.5-23g} h(E)\Big)^r\geq  \Big( c_0 \log 12h(E)\Big)^r,
\end{equation}
and the same explicit value of $c_0$ is valid. This concludes the proof.
\end{proof}

We will now close the discussion with the proof of Corollary \ref{Northcott}.

\begin{proof}
We split into two cases: in Theorem \ref{reg>rank}, either $c_0\log 12h(E)\leq 1$, in that case $p^s$ bounded implies a bounded height, or $c_0\log 12h(E)> 1$: in that case, as soon as the rank $r$ is positive and as long as $s$ is bounded from above, a bounded regulator implies a bounded height by inequality (\ref{minoration}). In both cases, apply \cite{MB85} Th\'eor\`eme 4.6 page 236, which proves that a bounded height implies finiteness, as the constant field is a finite field here.
\end{proof}

\end{document}